\documentclass[review,a4paper]{elsarticle}
\usepackage[margin=1in]{geometry}
\usepackage{lineno,hyperref}
\usepackage{dsfont}
\usepackage{tikz}
\usepackage{amsmath,amssymb,amsfonts,amsthm}

\usepackage{algorithm}
\usepackage{bm}
\usepackage{bookmark}

\newtheorem{thm}{Theorem}[section]
\newtheorem{lem}[thm]{Lemma}
\newtheorem{proposition}{Proposition}[section]

\theoremstyle{definition}
\newtheorem{dfn}{Definition}[section]
\newdefinition{con}{Construction}[section]
\newdefinition{rmk}{Remark}[section]

\theoremstyle{remark}
\numberwithin{equation}{section}

\journal{Journal of \LaTeX\ Templates}

\makeatletter
\def\ps@pprintTitle{%
   \let\@oddhead\@empty
   \let\@evenhead\@empty
   \def\@oddfoot{\reset@font\hfil\thepage\hfil}
   \let\@evenfoot\@oddfoot
}
\makeatother

\begin{document}

\begin{frontmatter}

\title{Directed Strongly Regular Cayley Graphs on Dihedral groups}
\author[abc,rvt]{Yiqin He\corref{cor1}}
\ead{2014750113@smail.xtu.edu.cn}
\author[rvt]{Bicheng Zhang\corref{cor1}}
\ead{zhangbicheng@xtu.edu.cn}
\author[abc]{Rongquan Feng}
\ead{fengrongquan@pku.edu.cn}
\address[abc]{School of Mathematics Science, Peking University, Beijing, 100877, PR China}
\address[rvt]{School of Mathematics and Computational Science, Xiangtan Univerisity, Xiangtan, Hunan, 411105, PR China}
\cortext[cor1]{Corresponding author}



%

\begin{abstract}
In this paper,\;we construct some  directed strongly regular Cayley graphs on Dihedral groups,\;which generalizes some earlier constructions.\;We also characterize certain directed strongly regular Cayley graphs on Dihedral groups $D_{p^\alpha}$,\;where $p$ is a prime and $\alpha\geqslant 1$ is a positive integer.\;
\end{abstract}
\end{frontmatter}
\section{Introduction}
A \emph{directed strongly regular graph} (DSRG) with parameters $( n, k, \mu ,\lambda , t)$ is a $k$-regular directed graph on $n$ vertices such that every vertex is on $t$ 2-cycles,\;and the number of paths of length two from a vertex $x$ to a vertex $y$ is $\lambda$ if there is an edge directed from $x$ to $y$ and  is $\mu$ otherwise.\;A DSRG with $t=k$ is an \emph{(undirected) strongly regular graph} (SRG).\;Duval showed that DSRGs with $t=0$ are the \emph{doubly regular tournaments}.\;It is therefore usually assumed that $0<t<k$.\;The DSRGs which satisfy the condition $0<t<k$ are called genuine DSRGs.\;The DSRGs appeared on this paper are all genuine.

Let $D$ be a directed graph with $n$ vertices.\;Let $A=\mathbf{A}(D)$ denote the adjacency matrix of $D$,\;and let $I = I_n$ and
$J = J_n$ denote the $n\times n$ identity matrix and all-ones matrix,\;respectively.\;Then $D$ is a directed strongly regular graph with parameters $(n,k,\mu,\lambda,t)$ if
and only if (i) $JA = AJ = kJ$ and (ii) $A^2=tI+\lambda A+\mu(J-I-A)$.\;

The constructions of DSRGs is a significant problem and has long been concerned.\;There are many constructions of DSRGs.\;Some of the known constructions use quadratic residue \cite{A},\;Kronecker product \cite{A},\;block matrices \cite{AD,A},\;combinatorial block designs \cite{FI},\;coherent algebras\cite{FI,K1,K2},\;Cayley digraph \cite{FI,HO,K2},\;generalized Cayley digraph \cite{FE},\;semidirect product \cite{D},\;finite incidence structures \cite{OL,olmez2014some},\;finite geometries \cite{FI},\;$1\frac{1}{2}$-designs \cite{BR},\;generalized quadrangle \cite{GE},\;BIBD \cite{GE},\;partial geometry \cite{GE,OL},\;double Paley designs \cite{FI},\;group divisible \cite{OL},\;difference digraph,\;partial sum families \cite{NE} and equitable partition \cite{EQ}.\;

Let $G$ be a finite (multiplicative) group and $S$ be a subset of $G\setminus\{e\}$.\;The \emph{Cayley graph} of $G$ generated by $S$,\;denoted by $\mathbf{Cay}(G,S)$,\;is the digraph $\Gamma$ such that $V(\Gamma)=G$ and $x\rightarrow y$ if and only if
$x^{-1}y\in S$,\;for $x,\;y\in G$.

Let $C_n=\langle x\rangle$ be a cyclic multiplicative group of order $n$.\;The \emph{dihedral group} $D_n$ is the group of symmetries of a regular $n$-polygon,\;and it can be viewed as a semidirect product of two cyclic groups $C_n=\langle x\rangle$ of order $n$ and $C_2=\langle a\rangle$ of order $2$.\;The presentation of $D_n$ is $D_n=C_n\rtimes C_2=\langle x,a|x^n=1,a^2=1,ax=x^{-1}a\rangle$.\;The cyclic group $C_n$ is a normal subgroup of $D_n$ of index $2$.\;

In this paper,\;we focus on the directed strongly regular Cayley graphs on dihedral groups.\;The \emph{Cayley graphs} on dihedral groups are called \emph{dihedrants}.\;A dihedrant which is a DSRG is called \emph{directed strongly regular dihedrant}.\;We now give some known directed strongly regular dihedrants.
\begin{thm}(\cite{K2})\label{t-1.1}Let $n$ be odd and let $X,Y\subset C_n$ satisfy the following conditions:\\
(i)\;$\overline{X}+\overline{X^{(-1)}}=\overline{C_n}-e$,\\
(ii)\;$\overline{Y}\;\overline{Y^{(-1)}}-\overline{X}\;\overline{X^{(-1)}}=\varepsilon\overline{C_n}$,$\varepsilon\in\{0,1\}$.\\
Then $\mathbf{Cay}(D_n,X\cup aY)$ is a DSRG with parameters $(2n,n-1+\varepsilon,\frac{n-1}{2}+\varepsilon,\frac{n-3}{2}+\varepsilon,\frac{n-1}{2}+\varepsilon)$.\;In particular,\;if $X$ satisfies $(i)$ and $Y=Xg$ or $X^{(-1)}g$ for some $g\in C_n$,\;then $\mathbf{Cay}(D_n,X\cup aY)$ is a DSRG with parameters $(2n,n-1,\frac{n-1}{2},\frac{n-3}{2},\frac{n-1}{2})$.\;
\end{thm}
We can say more when $n$ is an odd prime.
\begin{thm}(\cite{K2})\label{t-1.2}Let $n$ be an odd prime and let $X,Y\subset C_n$ and  $b\in D_n\setminus C_n$,\;Then the Cayley graph $\mathbf{Cay}(D_n,X\cup bY)$
 is a DSRG if and only if $X,Y$ satisfy the conditions of Theorem \ref{t-1.1}.\;
\end{thm}

\begin{thm}(\cite{FI})\label{t-1.3}Let $n$ be even,\;$c\in C_n$ be an involution and let $X,Y\subset C_n$ such that:\\
(i)\;$\overline{X}+\overline{X^{(-1)}}=\overline{C_n}-e-c$,\\
(ii)\;$\overline{Y}=\overline{X}$ or $\overline{Y}=\overline{X^{(-1)}}$,\\
(ii)\;$\overline{Xc}=\overline{X^{(-1)}}$.\;\\
Let $b\in D_n\setminus C_n$,\;then the Cayley graph $\mathbf{Cay}(D_n,X\cup bY)$
is a DSRG with parameters $(2n,n-1,\frac{n}{2}-1,\frac{n}{2}-1,\frac{n}{2})$.\;
\end{thm}

The following notataion will be used.\;Let $A$ be a multiset together with a \emph{multiplicity function} $\Delta_{A}$,\;where $\Delta_{A}(a)$ counting how many times $a$ occurs in the multiset $A$.\;We say $x$ belongs to $A$\;(i.e.\;$x\in A$)\;if $\Delta_A(x)>0$.\;
In the following,\;$A$ and $B$ are multisets,\;with multiplicity functions $\Delta_{A}$ and $\Delta_{B}$.\;

\begin{itemize}
  \item \textbf{Union,\;$A\uplus B$}:\;the union of multisets $A$ and $B$,\;is defined by $\Delta_{A\uplus B}=\Delta_A+\Delta_B$;
  \item \textbf{Scalar multiplication,\;$n\oplus A$}:\;the scalar multiplication of a multiset $A$ by a natural number $n$ by,\;is defined by $\Delta_{n\oplus A}=n\Delta_A$.
  \item \textbf{Difference,\;$A\setminus B$}:\;the difference of multisets $A$ and $B$,\;is defined by $\Delta_{A\setminus B}(x)=\max\{\Delta_A(x)-\Delta_B(x),0\}$ for any $x\in A$.\;
\end{itemize}
If $A$ and $B$ are usual sets,\;we use $A\cup B$,\;$A\cap B$ and  $A\setminus B$ denote the usual union,\;intersection and difference of $A$ and $B$.\;For example,\;if $A=\{1,2\}$ and $B=\{1,3\}$,\;then $A\uplus B=\{1,1,2,3\}$,\;$A\cup B=\{1,2,3\}$,\;$2\oplus A=\{1,1,2,2\}$,\;$A\setminus B=\{2\}$ and $\{1,1,2,2\}\setminus B=\{1,2,2\}$.\;

Throughout this paper,\;$n$ will denote a fixed positive integer,\;$\mathbb{Z}_n$ is the modulo $n$ residue class ring.\;For
a positive divisor $r$ of $n$ let $r\mathbb{Z}_n$ be the subgroup of the additive group of $\mathbb{Z}_n$ of order $\frac{n}{r}$.\;Let $v$ be a positive divisor of $n$ and $v\mathbb{Z}_n$ is a subgroup of $\mathbb{Z}_n$.\;Let $\pi_v$ be the natural projection from $\mathbb{Z}_n$ to the quotient group $\mathbb{Z}_n/v\mathbb{Z}_n$.\;It is clear that quotient group $\mathbb{Z}_n/v\mathbb{Z}_n$ is isomorphic to the cyclic group $\mathbb{Z}_v$.\;Let $\sigma_v$ be canonical isomorphism map between $\mathbb{Z}_v$ and $\mathbb{Z}_n/v\mathbb{Z}_n$.\;Let $\psi_v=\sigma_v\circ\pi_v$.\;Then $\psi_v$ is a canonical homomorphism
\begin{equation}\label{2.8}
\psi_{n,v}:\mathbb{Z}_n\rightarrow \mathbb{Z}_v,\;i+n\mathbb{Z}\mapsto \sigma_v(\pi_v(i))=i(mod\;v)+v\mathbb{Z}.
\end{equation}
with $\ker \psi_{n,v}=v\mathbb{Z}_n$.\;We also define $\phi_v:\mathbb{Z}_v\rightarrow\frac{n}{v}\mathbb{Z}_n,\;i+v\mathbb{Z}\rightarrow \frac{n}{v}i+n\mathbb{Z}$.\;It is clearly that $\phi_v$ gives
an isomorphism between $\mathbb{Z}_v$ and $\frac{n}{v}\mathbb{Z}_n$.

For a multisubset $A$ of $\mathbb{Z}_n$,\;let $$x^A=\biguplus_{i\in A}\Delta_A(i)\oplus\{x^i\},\;$$
then $C_n=x^{\mathbb{Z}_n}$ and $x^A$ is a multisubset of $C_n$.\;Let $A_1,A_2$ be the mutlisubsets of $\mathbb{Z}_n$ and $i\in \mathbb{Z}_n$,\;let $iA_1=\{ia_1|a_1\in A_1\}$,\;$i+A_1=\{i+a_1|a_1\in A_1\}$,\;and $i-A_1=\{i-a_1|a_1\in A_1\}$.\;The sum of mutlisubsets $A_1$ and $A_2$ is $A_1+A_2=\{a_1+a_2|a_1\in A_1,a_2\in A_2\}$,\;and the multiplicity function of $A_1+A_2$ is
\[\Delta_{A_1+A_2}(c)=\sum_{a_1+a_2=c}\Delta_{A_1}(a_1)\Delta_{A_2}(a_2)\]
for any $c\in \mathbb{Z}_n$.\;

Hence each subset $S$ of the dihedral group $D_n=\langle x,a|x^n=1,a^2=1,ax=x^{-1}a\rangle$ can be written in the form $S=x^X\cup (x^Ya)$ for some subsets $X,Y\subseteq \mathbb{Z}_n$,\;we denote the Cayley graph $\mathbf{Cay}(D_n, x^X\cup x^Ya)$ by $Dih(n,X,Y)$.\;Let $Dih(n,X,Y)$ be a directed strongly regular dihedrant.\;We list some properties in the following.\;
\begin{itemize}
  \item Its complement $Dih(n,\mathbb{Z}_{n}\setminus \{0\}\setminus X,\mathbb{Z}_{n}\setminus Y)$,\;is also a DSRG.\;
  \item Observe that $Dih(n,X,Y)$ is a digraph without loops,\;then $0\not\in X$,\;$X\neq -X$.
  \item The Lemma 2.5 in \cite{MI} asserts that the dihedrants $Dih(n,bX,b'+bY)$ and $Dih(n,X,Y)$ are isomorphic for any $b\in \mathbb{Z}_n^{\ast}$,\;$b'\in\mathbb{Z}_n$.\;Hence,\;we can also assume $0\in Y$.\;
\end{itemize}

For an odd prime $p$,\;the characterization of directed strongly regular dihedrant $Dih(p;X,Y)$ has been achieved in \cite{K2},\;which was presented in Theorem \ref{t-1.2}.\;In this paper,\;we characterize some certain directed strongly regular dihedrants $Dih(p^\alpha,X,Y)$ for prime power $p^\alpha$.\;We obtain the following results.\;

\begin{thm}\label{t-1.4}For an odd prime $p$ and a positive integer $\alpha$,\;then the dihedrant $Dih(p^\alpha,X,X)$ is a DSRG if and only if
$X=\psi_{\gamma}^{-1}(H)$ for some $1\leqslant\gamma\leqslant\alpha$ and a subset $H\subseteq \mathbb{Z}_{p^\gamma}$ satisfying the following conditions:\\
(i)\;$H\cup (-H)=\mathbb{Z}_{p^\gamma}\setminus\{0\}$.\\
(ii)\;$H\cap (-H)=\emptyset$.
\end{thm}

\begin{thm}\label{t-1.5}For a positive integer $\alpha$,\;the dihedrant $Dih(2^\alpha,X,X)$ is a DSRG if and only if
$X=\psi_{\gamma}^{-1}(H)$ for some $2\leqslant\gamma\leqslant\alpha$ and a subset $H\subseteq \mathbb{Z}_{p^\gamma}$ satisfying the following conditions:\\
(i)\;$H\cup (-H)=(\mathbb{Z}_{p^\gamma}\setminus\{0\})\uplus\{2^{\gamma-1}\}$.\\
(ii)\;$H\cap (2^{\gamma-1}+H)=\emptyset$.
\end{thm}

\begin{thm}\label{t-1.6}The dihedrant $Dih(p^\alpha,X,Y)$ is a DSRG with $X\subset Y$ and $Y\setminus X$ is a union of some $\mathbb{Z}_{p^\alpha}^\ast$-orbits if and only if $Y=\psi_{\gamma}^{-1}(H)$,\;$X=Y\setminus\{0\}$ or $X=Y\setminus p^\gamma\mathbb{Z}_{p^\alpha}$
for some $1\leqslant\gamma\leqslant\alpha$ and a subset $H\subseteq \mathbb{Z}_{p^\gamma}$ such that $H\uplus(-H)=\mathbb{Z}_{p^\gamma}\uplus \{0\}$.\;
\end{thm}
\section{Preliminary}

\subsection{Properties of DSRG}
Duval \cite{A} developed necessary conditions on the parameters of $(n,k,\mu,\lambda,t)$-DSRG and calculated the spectrum of a DSRG.\;
\begin{proposition}(see \cite{A})\label{p-2.1}
A DSRG with parameters $( n,k,\mu ,\lambda ,t)$ with $0<t<k$ satisfy
\begin{equation}\label{2.1}
k(k+(\mu-\lambda ))=t+\left(n-1\right)\mu,
\end{equation}
\[{{d}^{2}}={{\left( \mu -\lambda  \right)}^{2}}+4\left(t-\mu\right)\text{,}d|2k-(\lambda-\mu)(n-1),\]
\[\frac{2k-(\lambda-\mu)(n-1)}{d}\equiv n-1(mod\hspace{2pt}2)\text{,}\left|\frac{2k-(\lambda-\mu)(n-1)}{d}\right|\leqslant n-1,\]
where d is a positive integer, and
$$ 0\leqslant \lambda <t<k,0<\mu\leqslant t<k,-2\left( k-t-1 \right)\le \mu -\lambda \leqslant 2\left( k-t \right).$$
\end{proposition}
\begin{rmk}If $G$ is a DSRG with parameters $( n,k,\mu ,\lambda ,t)$ with $0<t=\mu<k$,\;then from the last inequality in the above proposition ,\;we have $\lambda-\mu<0$.\;
\end{rmk}
\begin{proposition}(see \cite{A})\label{p-2.2}
A DSRG with parameters $( n, k, \mu ,\lambda , t)$ has three distinct integer eigenvalues
$$ k>\rho =\frac{1}{2}\left( -\left( \mu -\lambda  \right)+d \right)>\sigma =\frac{1}{2}\left( -\left( \mu -\lambda  \right)-d \right),\; $$
The multiplicities are
$$  1,\;m_\rho=-\frac{k+\sigma \left( n-1 \right)}{\rho -\sigma }\text{\;and\;}m_\sigma=\frac{k+\rho \left( n-1 \right)}{\rho -\sigma },\;$$
respectively.\;
\end{proposition}
\begin{proposition}(see \cite{A})\label{p-2.3}
If $G$ is a DSRG with parameters $( n, k, \mu ,\lambda , t)$,\;then the complementary $G'$ is also a DSRG with parameters $(n',k',\mu',\;\lambda' ,t')$,\;where $k'=(n-2k)+(k-1)$,\;$\lambda'=(n-2k)+(\mu-2)$,\;$t'=(n-2k)+(t-1)$,\;$\mu'=(n-2k)+\lambda$.\;
\end{proposition}

\begin{dfn}\label{d-2}(Group Ring)\\
For any finite (multiplicative) group $G$ and ring $R$,\;the \emph{Group Ring} $R[G]$ is defined as the set of all formal sums of elements of $G$,\;with coefficients from $R$.\;i.e.,
\[R[G]=\left\{\sum_{g\in G}r_gg|r_g\in R,\; r_g\neq 0\text{ for finite g}\right\}.\;\]
The operations $+$ and $\cdot$ on $R[G]$ are given
by
\[\sum_{g\in G}r_gg+\sum_{g\in G}s_gg=\sum_{g\in G}(r_g+s_g)g,\;\]
\[\left(\sum_{g\in G}r_gg\right)\cdot\left(\sum_{g\in G}s_gg\right)=\left(\sum_{g\in G}t_gg\right),\;t_g=\sum_{g'g''=g}r_{g'}s_{g''}.\]
\end{dfn}

For any multisubset $X$ of $G$,\;Let $\overline{X}$ denote the element of the group ring $R[G]$ that is the sum of all elements of $X$,\;i.e., \[\overline{X}=\sum_{x\in X}\Delta_X(x)x.\]

The lemma below allows us to express a sufficient and necessary condition for a Cayley graph to be  strongly regular in terms of group ring.
\begin{lem}\label{l-2.1}
The Cayley graph $\mathbf{Cay}(G, S)$ of $G$ with respect to $S$ is a DSRG with parameters $(n, k, \mu,  \lambda, t)$ if and only if $|G| = n$, $|S|= k$, and
\[\overline{S}^2=te+\lambda\overline{S}+\mu(\overline{G}-e-\overline{S}).\]
\end{lem}
A \emph{character} $\chi$ of a finite abelian group is a homomorphism from $G$ to $\mathbb{C}^*$,\;the multiplicative group of $\mathbb{C}$.\;All characters of $G$ form a group under the multiplication $\chi\chi'(a)=\chi(a)\chi'(a)$ for any $a\in G$,\;which is denoted by $\widehat{G}$ and it is called the \emph{character group} of ${G}$.\;It is easy to see that $\widehat{G}$ is isomorphic to $G$.\;Every character $\chi\in\widehat{G}$ of $G$ can be extended to a homomorphism from $\mathbb{C}[G]$ to $\mathbb{C}$  by\;
\[\chi(a)=\sum_{g\in G}a_g\chi(g),\text{\;for\;} a=\sum_{g\in G}a_gg\in\mathbb{C}[G].\]

Let $\omega$ be a fixed complex primitive $n$-th root of unity.\;Then
$\widehat{C_n}=\{\chi_j|j\in \mathbb{Z}_n\}$,\;where $\chi_j(x^i)=\omega^{ij}$ for $0\leqslant i,j\leqslant n-1.$
\subsection{Fourier Transformation}
Throughout this section $n$ will denote a fixed positive integer,\;$\mathbb{Z}_n$ is the modulo $n$ residue class ring.\;For
a positive divisor $r$ of $n$ let $r\mathbb{Z}_n$ be the subgroup of the additive group of $\mathbb{Z}_n$ of order $\frac{n}{r}$.\;

The following statement and  notations coincide with \cite{MI}.\;
Let $\mathbb{Z}_n^{\ast}$  be the multiplicative group
of units in the ring $\mathbb{Z}_n$.\;Then $\mathbb{Z}_n^{\ast}$ has an action on $\mathbb{Z}_n$ by multiplication and hence $\mathbb{Z}_n$ is a union of some  $\mathbb{Z}_n^{\ast}$-orbits.\;Each $\mathbb{Z}_n^{\ast}$-orbit
consists of all elemnents of a given order in the additive group $\mathbb{Z}_n$.\;If $r$ is a positive divisor of $n$,\;we denote  the
$\mathbb{Z}_n^{\ast}$-orbit containing all elements of order $r$ by $\mathcal{O}_r$.\;Thus
\[\mathcal{O}_r=\left\{z\bigg|z\in \mathbb{Z}_n,\;\frac{n}{(n,z)}=r\right\}=\left\{c{\frac{n}{r}}\bigg|1\leqslant c\leqslant r,(r,c)=1\right\}\]
and $|\mathcal{O}_r|=\varphi(r)$.

Let $\zeta_n$ be a fixed primitive $n$-th root of unity and $\mathbb{F}=\mathbb{Q}[\zeta_n]$ the $n$-th
\emph{Cyclotomic Field} over $\mathbb{Q}$.\;Further,\;let $\mathbb{F}^{\mathbb{Z}_n}$ be the $\mathbb{F}$-vector space of all
functions $f:\mathbb{Z}_n\rightarrow\mathbb{F}$ mapping from the residue class ring $\mathbb{Z}_n$ to the field $\mathbb{F}$ (with the
scalar multiplication and addition defined point-wise).\;The $\mathbb{F}$-algebra obtained from $\mathbb{F}^{\mathbb{Z}_n}$ by
defining the multiplication point-wise will be denoted by $(\mathbb{F}^{\mathbb{Z}_n},\cdot)$.\;The $\mathbb{F}$-algebra obtained from $\mathbb{F}^{\mathbb{Z}_n}$ by defining the multiplication as \emph{convolution} will be denoted by $(\mathbb{F}^{\mathbb{Z}_n},\ast)$,\;where the convolution is defined by:\;$(f\ast g)(z)=\sum_{i\in \mathbb{Z}_n}f(i)g(z-i)$.\;

We recall that $cA=\{ca|a\in A\}$,\;$i+A=\{i+a|a\in A\}$,\;$i-A=\{i-a|a\in A\}$,\;$A+B=\{a+b|a\in A,\;b\in B\}$.\;The \emph{Fourier transformation} which is an isomorphsim of $\mathbb{F}$-algebra $(\mathbb{F}^{\mathbb{Z}_n},\ast)$ and $(\mathbb{F}^{\mathbb{Z}_n},\cdot)$ is defined by
\begin{equation}\label{2.2}
\mathcal{F}:(\mathbb{F}^{\mathbb{Z}_n},\ast)\rightarrow (\mathbb{F}^{\mathbb{Z}_n},\cdot),\;\;\;\;(\mathcal{F}f)(z)=\sum_{i\in \mathbb{Z}_n}f(i)\zeta_n^{iz}.
\end{equation}
Then,\;$\mathcal{F}(f\ast g)=\mathcal{F}(f)\mathcal{F}(g)$ for $f,g\in \mathbb{F}^{\mathbb{Z}_n}$.\;It also obeys the \emph{inversion formula}
\begin{equation}\label{2.3}
\mathcal{F}(\mathcal{F}(f))(z)=nf(-z).
\end{equation}
Let $A$ and $B$ be multisubsets of $\mathbb{Z}_n$,\;we have
\begin{equation}\label{2.4}
\mathcal{F}\Delta_{(-A)}=\overline{\mathcal{F}\Delta_{A}},\;\mathcal{F}\Delta_{A+B}=\mathcal{F}(\Delta_A\ast\Delta_B)=(\mathcal{F}\Delta_{A})(\mathcal{F}\Delta_{B}).
\end{equation}
Recall that $\widehat{C_n}=\{\chi_j|\;j\in \mathbb{Z}_n\}$,\;we also have
\begin{equation}\label{2.5}
(\mathcal{F}\Delta_A)(z)=\chi_z(\overline{A}).\;
\end{equation}
The Fourier transform of characteristic functions of  additive subgroups in $\mathbb{Z}_n$  can be easily computed.\;For a positive divisor $r$ of $n$,\;
\begin{equation}\label{2.6}
\mathcal{F}\Delta_{r\mathbb{Z}_n}=\frac{n}{r}\Delta_{\frac{n}{r}\mathbb{Z}_n},\;\mathcal{F}\Delta_{\mathbb{Z}_n}=n\Delta_{0},\;\text{and\;}\mathcal{F}\Delta_{0}=\Delta_{\mathbb{Z}_n}=1.
\end{equation}
The following lemma will be used in this paper.
\begin{lem}(\cite{MI})\label{l-2.2}Suppose that $f:\mathbb{Z}_n\rightarrow\mathbb{F}$ is a function such that ${\rm{Im}}(f)\subseteq \mathbb{Q}$.\;Then ${\rm{Im}}(\mathcal{F}f)\subseteq \mathbb{Q}$  if and only if $f=\sum_{r|n}\alpha_r\Delta_{\mathcal{O}_r}$ for some $\alpha_r\in\mathbb{Q}$.
\end{lem}
The value of Fourier Transformation of the characteristic function of an orbit $\mathcal{O}_r$ is also known as the Ramanujan's sum,\;i.e.,\;
\begin{equation}\label{2.7}
(\mathcal{F}\Delta_{\mathcal{O}_r})(z)={\hat{\mu}}\left(\frac{r}{(r,z)}\right)\frac{\varphi(r)}{\varphi\left(\frac{r}{(r,z)}\right)}\in\mathbb{Z}.
\end{equation}
where ${\hat{\mu}}$ is M\"{o}bius function.\;

Let $r$ be a positive divisor of $n$,\;and $H$ is a multisubset of $\mathbb{Z}_v$,\;then $\psi_{n,v}^{-1}(H)$ is a multisubset of $\mathbb{Z}_n$.\;In other world,\;$\psi_{n,v}^{-1}(H)=S+v\mathbb{Z}_n$ for some multisubset $S$ of $\{0,1,\ldots,v-1\}$.\;
Let ${\Delta}^{(v)}_{H}$ be multiplicity function of $H$ in $\mathbb{Z}_v$,\;then for any $z\in \mathbb{Z}_v$,\;
\begin{equation}\label{identity}
\begin{aligned}
(\mathcal{F}{\Delta}^{(v)}_{H})(z)=\frac{n}{v}\left(\mathcal{F}{\Delta}_{\psi_{n,v}^{-1}(H)}(\phi_v(z))\right).\;
\end{aligned}
\end{equation}

\section{Some lemmas}
For $z\in \mathbb{Z}_{p^\alpha}$,\;define $\nu_p(z)$ as the maximum  power of the prime $p$ that divides $z$.\;Note that all the  $\mathbb{Z}_{p^\alpha}^\ast$-orbits are $\mathcal{O}_{1}=\mathcal{O}_{p^0},\mathcal{O}_{p},\mathcal{O}_{p^2},\ldots$ and $\mathcal{O}_{p^\alpha}$,\;where
\[\mathcal{O}_{p^z}=\left\{cp^{\alpha-z}|1\leqslant c\leqslant p^z,(p,c)=1\right\}=\{i|i\in\mathbb{Z}_{p^\alpha},\nu_p(i)=\alpha-z\}.\]
In particular,\;$\mathcal{O}_{1}=\{0\}=\{i|i\in\mathbb{Z}_{p^\alpha},\nu_p(i)=\alpha\}$.\;We denote $\mathcal{O}_{p^i}$ with $O_{i}$ for simplicity throughout this paper.\;Note that $p^{\alpha-\beta}\mathbb{Z}_{p^\alpha}$ is a subgroup of $\mathbb{Z}_{p^\alpha}$ and
\begin{equation}\label{3.1}
p^{\alpha-\beta}\mathbb{Z}_{p^\alpha}=\bigcup_{i=0}^\beta{O}_{i}.
\end{equation}
for each $0\leqslant \beta\leqslant\alpha$.\;Let $\widehat{C_{p^\alpha}}=\{\chi_z|z\in\mathbb{Z}_{p^\alpha}\}$,\;then the order of the character $\chi_z$ in the character group $\widehat{C_{p^\alpha}}$  is $p^{\alpha-\nu_p(z)}$.\;

In this section,\;let $X$ be a subset of $\mathbb{Z}_{p^\alpha}$,\;let $U_X=X\uplus(-X)$ and 
\[\mathbf{q}(z)\overset{\text{def}}=(\mathcal{F}\Delta_{U_X})(z),\]
for $z\in \mathbb{Z}_{p^\alpha}$.\;For any $\delta\in \mathbb{C}$,\;define $\Gamma_{\delta}=\{z\in \mathbb{Z}_{p^\alpha}|\;\mathbf{q}(z)=(\mathcal{F}\Delta_{U_X})(z)=\delta\}$.\;We also define one condition of $S$ by
\begin{equation}\label{assumption}
0\not\in X,\;X\neq -X,\;\text{and\;}(\mathcal{F}\Delta_{X\uplus(-X)})(z)=(\mathcal{F}\Delta_{X})(z)+\overline{(\mathcal{F}\Delta_{X})(z)}\in\{0,-m\},\;\forall \;0\neq z\in \mathbb{Z}_{p^\alpha},
\end{equation}
where $m$ is a positive integer.\;Therefore,\;if $S$ satisfies the condition $(\ref{assumption})$,\;we let $\Gamma=\Gamma_{-m}$,\;and we can obtain $\mathbf{q}(z)\leqslant0$ for any $z\neq 0$.\;Thus
\begin{equation}\label{cor1}
\mathbf{q}(z)>0 \text{\;implies that\;} z=0.
\end{equation}
If for some $z_1$ and $z_2$,\;we have
\begin{equation}\label{cor2}
\mathbf{q}(z_1),\mathbf{q}(z_2)<0,\;\text{then\;}\mathbf{q}(z_1)=\mathbf{q}(z_2)=-m.
\end{equation}

\begin{lem}\label{l-3.1}Let $X$ be a subset of $\mathbb{Z}_{p^\alpha}$ which satisfies $\ref{assumption}$.\;Then there are some integers $1\leqslant r_1<r_2<\ldots<r_s\leqslant \alpha$ and $1\leqslant r_{s+1}<r_{s+2}<\ldots<r_{t}\leqslant \alpha$ such that
\begin{equation}\label{3.4}
U_X=(2\otimes {O}_{r_1})\uplus (2\otimes {O}_{r_2})\uplus\ldots\uplus (2\otimes {O}_{r_s})\uplus ({O}_{r_{s+1}}\cup {O}_{r_{s+2}}\cup\ldots\cup {O}_{r_t}),
\end{equation}
\end{lem}
\begin{proof}Note that $$\mathrm{Im}(\mathbf{q})\in \mathbb{Q}.$$\;Therefore,\;we have
\[\Delta_{U_X}=\sum_{r|p^\alpha}\alpha_r\Delta_{\mathcal{O}_r}=\sum_{r=0}^\alpha\alpha_r\Delta_{{O}_r}\]
for some $\alpha_r\in\{0,1,2\}$ by Lemma \ref{l-2.1}.\;Note that $0\not\in X$,\;so $\alpha_0=0$.\;This shows that there are some intergers $1\leqslant r_1<r_2<\ldots<r_s\leqslant \alpha$ and $1\leqslant r_{s+1}<r_{s+2}<\ldots<r_{t}\leqslant \alpha$ such that \ref{3.4} holds.
\end{proof}

Throughout this section,\;we let $\mathcal{I}_1=\{r_1,r_2,\ldots,r_s\}$ and $\mathcal{I}_2=\{r_{s+1},r_{s+2},\ldots,r_{t}\}$,\;then $\mathcal{I}_2\neq \emptyset$ from $X\neq -X$.\;Thus $|\mathcal{I}_1|=s$ and $|\mathcal{I}_2|=t-s\geqslant1$.\;

\begin{lem}\label{l-3.2}Let $\mathcal{I}_1$ and $\mathcal{I}_2$ be sets defined above.\;Then $\mathcal{I}_1$ and $\mathcal{I}_2$ form a partition of $\{\beta+1,\beta+2,\ldots,\alpha\}$,\;where $\beta=\min\{v_p(z)|\;z\in \Gamma\}$.\;Moreover,\;we have
\begin{equation}\label{3.5}
\begin{aligned}
U_X=(O_{r_1}\cup O_{r_2}\cup\ldots\cup O_{r_s})\uplus\left(\mathbb{Z}_{p^\alpha}\setminus p^{\alpha-\beta}\mathbb{Z}_{p^\alpha}\right).
\end{aligned}
\end{equation}
\end{lem}
\begin{proof}By the  inversion formula (\ref{2.3}) of Fourier transformation,\;we have
\begin{equation*}
\begin{aligned}
p^\alpha\Delta_U(z)&=(\mathcal{F}(\mathcal{F}\Delta_U))(-z)=-m\sum_{i\in \Gamma}\omega^{-iz}+|U|.\\
\end{aligned}
\end{equation*}
Note that $0\not\in U$ and hence $\Delta_U(0)=-m|\Gamma|+|U|=0$.\;Thus we can get that
\begin{equation*}
\begin{aligned}
z\in\mathbb{Z}_{p^\alpha}\setminus U &\Leftrightarrow\Delta_U(z)=0\Leftrightarrow  \Delta_U(z)-\Delta_U(0)=0\Leftrightarrow \left(\sum_{i\in \Gamma}\omega^{-iz}+|U|\right)-(|\Gamma|+|U|)=0\\
&\Leftrightarrow \sum_{i\in \Gamma}(1-\omega^{-iz})=0 \Leftrightarrow \omega^{-iz}=1,\forall i\in\Gamma\Leftrightarrow z\in \bigcap_{i\in\Gamma}\frac{p^\alpha}{(p^\alpha,i)}\mathbb{Z}_{p^\alpha}.
\end{aligned}
\end{equation*}
This gives that
\[\mathbb{Z}_{p^\alpha}\setminus U=\bigcap_{i\in\Gamma}\frac{p^\alpha}{(p^\alpha,i)}\mathbb{Z}_{p^\alpha}=p^{\alpha-\beta}\mathbb{Z}_{p^\alpha},\]
where $\beta$ is defined as above.\;This completes the proof.\;
\end{proof}


\begin{rmk}From the above lemma,\;we have $t=\alpha-\beta$.\;
\end{rmk}
We now prove the following lemma which asserts that $\mathcal{I}_1=\emptyset$ if $p$ is a odd prime.\;
\begin{lem}\label{l-3.3}Let $X\subset \mathbb{Z}_{p^\alpha}$ which satisfies $(\ref{assumption})$  and $\mathcal{I}_1$ and $\mathcal{I}_2$ be sets defined above.\;If $p$ be a odd prime,\;then $\mathcal{I}_1=\emptyset$ and
\begin{equation}\label{3.6}
\begin{aligned}
U_X=\mathbb{Z}_{p^\alpha}\setminus p^{\alpha-\beta}\mathbb{Z}_{p^\alpha}.
\end{aligned}
\end{equation}
\end{lem}
\begin{proof}This lemma holds for $\alpha=1$ clearly.\;We assume that $\alpha\geqslant 2$.\;Suppose that $\mathcal{I}_1\neq\emptyset$.\;Then $s=|\mathcal{I}_1|>0$.\;Since $X\neq(-X)$,\;then $\mathcal{I}_2\neq\emptyset$ and hence $\beta\leqslant\alpha-2$.\;It follows from Lemma \ref{l-3.2} and equations (\ref{2.6}),\;(\ref{2.7}) that
\begin{equation*}
\begin{aligned}
\mathbf{q}(z)&=\sum_{i=1}^s(\mathcal{F}\Delta_{O_{r_i}})(z)+\mathcal{F}\Delta_{\mathbb{Z}_{p^\alpha}}(z)-(\mathcal{F}\Delta_{p^{\alpha-\beta}\mathbb{Z}_{p^\alpha}})(z)\\
&=\sum_{i=1}^s\mu\left(\frac{p^{r_i}}{(p^{r_i},z)}\right)\frac{\varphi(p^{r_i})}{\varphi\left(\frac{p^{r_i}}{(p^{r_i},z)}\right)}
+p^\alpha\Delta_{0}(z)-p^\beta(\Delta_{p^{\beta}\mathbb{Z}_{p^\alpha}})(z).
\end{aligned}
\end{equation*}

We now divide into two cases.\;\\
{${\mathbf{Case\;1}}$}:\;$s\geq 2$.\;We assert that $r_{i+1}=r_{i}+1$ for $1\leqslant i\leqslant s-1$.\;Otherwise,\;there is an integer  $u$ with $1\leqslant u\leqslant s-1$ and  $r_{u+1}>r_{u}+1$.\;Thus
\begin{equation*}
\begin{aligned}
\mathbf{q}(p^{r_u})&=\sum_{i=1}^u\mu\left(\frac{p^{r_i}}{(p^{r_i},p^{r_u})}\right)\frac{\varphi(p^{r_i})}{\varphi\left(\frac{p^{r_i}}{(p^{r_i},p^{r_u})}\right)}
+p^\alpha\Delta_{0}(p^{r_u})-p^\beta(\Delta_{p^{\beta}\mathbb{Z}_{p^\alpha}})(p^{r_u})\\
&=\sum_{i=1}^up^{r_i-1}(p-1)-p^{\beta}>p^{r_1-1}(p-1)-p^{\beta}\geq p^{\beta}(p-1)-p^{\beta}>0,
\end{aligned}
\end{equation*}
contradicting  $(\ref{cor1})$.\;Thus $\mathcal{I}_1=\{r_1,r_2,\ldots,r_s\}=\{\gamma,\gamma+1,\ldots,\gamma+s-1\}$ for some $\gamma\geq\beta+1$.\;Note that
\begin{equation*}
\begin{aligned}
\mathbf{q}(p^{\gamma+s-1})&=\sum_{i=\gamma}^{\gamma+s-1}\mu\left(\frac{p^{i}}{(p^{i},p^{\gamma+s-1})}\right)\frac{\varphi(p^{i})}{\varphi\left(\frac{p^{i}}{(p^{i},p^{\gamma+s-1})}\right)}
-p^{\beta}\\
&=\sum_{i=\gamma}^{\gamma+s-1}\varphi(p^{i})-p^{\beta}=\sum_{i=\gamma}^{\gamma+s-1}p^{r_i-1}(p-1)-p^{\beta}\\
&=p^{\gamma+s-1}-p^{\gamma-1}-p^{\beta}=p^{\gamma-1}(p^s-1)-p^{\beta}\geq p^{\beta}(p^s-2)>0,
\end{aligned}
\end{equation*}
then $p^{\gamma+s-1}=0$ from $(\ref{cor1})$ and hence $\gamma+s-1=\alpha$,\;$\mathcal{I}_1=\{\alpha-s+1,\alpha-s+2,\ldots,\alpha\}$ and $\mathcal{I}_2=\{\beta+1,\beta+2,\ldots,\alpha-s\}$.\;Note that $\mathcal{I}_2\neq \emptyset$ and hence $\beta+1\in \mathcal{I}_2$,\;so $\alpha-s+1\geqslant\beta+2$.\;If $\alpha-s+1>\beta+2$,\;then
\begin{equation*}
\begin{aligned}
\mathbf{q}(p^{\beta+1})&=\mu\left(\frac{p^{\alpha-s+1}}{(p^{\alpha-s+1},p^{\beta+1})}\right)\frac{\varphi(p^{\alpha-s+1})}{\varphi\left(\frac{p^{\alpha-s+1}}{(p^{i},p^{\beta+1})}\right)}
-p^{\beta}=-p^{\beta},
\end{aligned}
\end{equation*}
but \begin{equation*}
\mathbf{q}(p^{\alpha-s+1})=-p^{\alpha-s}-p^{\beta}<-p^{\beta}=\mathbf{q}(p^{\beta+1}),
\end{equation*}
contradicting  $(\ref{cor2})$.\;This shows that $\alpha-s+1=\beta+2$,\;so $\mathcal{I}_1=\{\beta+2,\beta+3,\ldots,\alpha\}$ and $\mathcal{I}_2=\{\beta+1\}$.\;In this case,\;we can get $\mathbf{q}(p^{\beta})=-p^{\beta}$,\;but note that
\begin{equation*}
\mathbf{q}(p^{\beta+1})=\mu\left(\frac{p^{\beta+2}}{(p^{\beta+2},p^{\beta+1})}\right)\frac{\varphi(p^{\beta+2})}{\varphi\left(\frac{p^{\beta+2}}{(p^{\beta+2},p^{\beta+1})}\right)}-p^{\beta}=-(1+p)\cdot p^{\beta}\neq\mathbf{q}(p^{\beta}),
\end{equation*}
contradicting $(\ref{cor2})$.\;\\
{${\mathbf{Case\;2}}$}:\;$s=1$.\;Let $\mathcal{I}_1=\{\gamma\}$ with some $\alpha\geq\gamma\geq\beta+1$.\;Note that
\begin{equation*}
\mathbf{q}(p^{\gamma})=p^{\gamma-1}(p-1)-p^{\beta}>0,
\end{equation*}
then $q^\gamma=0$,\;then $\mathcal{I}_1=\{\alpha\}$.\;But
\[\mathbf{q}(p^{\alpha-2})=-p^{\beta}\neq \mathbf{q}(p^{\alpha-1})=-p^{\beta}-p^{\alpha-1},\]
which also leads to a contradiction to $(\ref{cor2})$.\;

Combining {${\mathbf{Case\;1}}$} ans {${\mathbf{2}}$},\;we can obtain that $\mathcal{I}_1=\emptyset$ and
\begin{equation*}
\begin{aligned}
U_X=\bigcup_{i=\beta+1}^\alpha O_i=\mathbb{Z}_{p^\alpha}\setminus p^{\alpha-\beta}\mathbb{Z}_{p^\alpha}.
\end{aligned}
\end{equation*}
from Lemma \ref{l-3.2}.
\end{proof}

We now consider $p=2$,\;and $\alpha\geqslant 2$,\;since we cannot find a subset $X$ of $\mathbb{Z}_2\setminus\{0\}=\{1\}$ such that $X\neq(-X)$.\;
\begin{lem}\label{l-3.4}Let $X\subset \mathbb{Z}_{2^\alpha}$ satisfies condition $(\ref{3.2})$  and $\mathcal{I}_1$ and $\mathcal{I}_2$ be sets defined above.\;If $\mathcal{I}_1\neq\emptyset$,\;then $\mathcal{I}_1=\{\beta+1\}$,\; $\mathcal{I}_2=\{\beta+1,\beta+2,\ldots,\alpha\}$ and
\begin{equation}\label{3.7}
\begin{aligned}
U_X=O_{\beta+1}\uplus\left(\mathbb{Z}_{2^\alpha}\setminus 2^{\alpha-\beta}\mathbb{Z}_{2^\alpha}\right).
\end{aligned}
\end{equation}
\end{lem}
\begin{proof}First,\;we show that $\mathcal{I}_1$ is a singleton set.\;Otherwise,\;we assume $s=|\mathcal{I}_1|\geqslant2$.\;Note that $\mathcal{I}_2\neq\emptyset$ and hence $\beta\leqslant\alpha-3$.\;From Lemma \ref{l-3.2} and equations (\ref{2.6}),\;(\ref{2.7}),\;we have
\begin{equation*}
\begin{aligned}
\mathbf{q}(z)=\sum_{i=1}^s\mu\left(\frac{2^{r_i}}{(2^{r_i},z)}\right)\frac{\varphi(2^{r_i})}{\varphi\left(\frac{2^{r_i}}{(2^{r_i},z)}\right)}
+2^\alpha\Delta_{0}(z)-2^\beta(\Delta_{2^{\beta}\mathbb{Z}_{2^\alpha}})(z).
\end{aligned}
\end{equation*}
We assert that $\beta+1\in \mathcal{I}_1$,\;i.e.,\;$r_1=\beta+1$.\;Otherwise we can assume $r_1\geqslant\beta+2$ and hence $\beta+1\in\mathcal{I}_2$.\;If there is a $u$ such that $r_{u+1}>r_{u}+1$ for some $1\leqslant u\leqslant s-1$,\;then
\begin{equation*}
\begin{aligned}
\mathbf{q}(2^{r_u})
&=\sum_{i=1}^u2^{r_i-1}-2^{\beta}\geqslant2^{r_1-1}-2^{\beta}\geqslant 2^{\beta+1}-2^{\beta}=2^{\beta}>0,
\end{aligned}
\end{equation*}
contradicting $(\ref{cor1})$.\;So $r_{i+1}=r_{i}+1$ for $1\leqslant i\leqslant s-1$ and then $\mathcal{I}_1=\{\gamma,\gamma+1,\ldots,\gamma+s-1\}$ for some $\gamma\geq\beta+2$.\;Note that
\begin{equation*}
\begin{aligned}
\mathbf{q}(2^{\gamma+s-1})&=\sum_{i=\gamma}^{\gamma+s-1}\mu\left(\frac{2^{i}}{(2^{i},2^{\gamma+s-1})}\right)\frac{\varphi(2^{i})}{\varphi\left(\frac{2^{i}}{(2^{i},2^{\gamma+s-1})}\right)}
-2^{\beta}\\
&=\sum_{i=\gamma}^{\gamma+s-1}\varphi(2^{i})-2^{\beta}=2^{\gamma+s-1}-2^{\gamma-1}-2^{\beta}\geqslant2^{\beta+1}(2^s-1)-2^{\beta}>0.
\end{aligned}
\end{equation*}
Then from $(\ref{cor1})$,\;$2^{\gamma+s-1}=0$,\;so $\gamma+s-1=\alpha$,\;$\mathcal{I}_1=\{\alpha-s+1,\alpha-s+2,\ldots,\alpha\}$ and $\mathcal{I}_2=\{\beta+1,\beta+2,\ldots,\alpha-s\}$.\;Since $\beta+1\in \mathcal{I}_2$,\;we have   $\alpha-s+1\geqslant\beta+2$.\;If $\alpha-s+1>\beta+2$,\;then
\begin{equation*}
\begin{aligned}
\mathbf{q}(2^{\beta+1})&=-2^{\beta}\neq\mathbf{q}(2^{\alpha-s+1})=-2^{\alpha-s}-2^{\beta},
\end{aligned}
\end{equation*}
contradicting  $(\ref{cor2})$ and hence $\alpha-s+1=\beta+2$.\;So $\mathcal{I}_1=\{\beta+2,\beta+3,\ldots,\alpha\}$ and $\mathcal{I}_2=\{\beta+1\}$.\;In this case,\;$\mathbf{q}(2^{\beta})=-2^{\beta}$,\;but
\begin{equation*}
\mathbf{q}(2^{\beta+1})=\mu\left(\frac{2^{\beta+2}}{(2^{\beta+2},2^{\beta+1})}\right)\frac{\varphi(2^{\beta+2})}{\varphi\left(\frac{2^{\beta+2}}{(2^{\beta+2},2^{\beta+1})}\right)}-2^{\beta}=-3\cdot2^{\beta}\neq\mathbf{q}(2^{\beta}),
\end{equation*}
contradicting  $(\ref{cor2})$.\;Therefore we have $\beta+1\in \mathcal{I}_1$.\;\\
{${\mathbf{Case\;1}}$}.\;$s=2$.\;In this case,\;we have
\[\mathcal{I}_1=\{\beta+1,\beta+\kappa\},\;\mathcal{I}_2=\{\beta+2,\ldots,\beta+\kappa-1,\beta+\kappa+1,\ldots,\alpha\}.\]
for some $\kappa$.\;Note that
\begin{equation*}
\begin{aligned}
\mathbf{q}(2^{\beta+\kappa})&=2^\beta+2^{\beta+\kappa-1}-2^{\beta}=2^{\beta+\kappa-1}>0,
\end{aligned}
\end{equation*}
then $2^{\beta+\kappa}= 0$ by $(\ref{cor1})$.\;Thus we have $\beta+\kappa=\alpha$ and hence $\mathcal{I}_1=\{\beta+1,\alpha\}$.\;Therefore,\; $\kappa=\alpha-\beta=t=s+(t-s)\geq 2+1=3$.\;But $\mathbf{q}(2^{\alpha-1})=-2^{\alpha-1}\neq\mathbf{q}(2^{\beta})=-2\cdot2^{\beta}=-2^{\beta+1}$,\;which leads to a contradiction to $(\ref{cor2})$.\\
{${\mathbf{Case\;2}}$}.\;$s>2$.\;Similar to the discussion of equation (\ref{5.5}),\;we can also obtain
\[\mathcal{I}_1=\{\beta+1,\beta+\kappa,\ldots,\beta+\kappa+s-2\}\;\text{and\;}\mathcal{I}_2=\{\beta+1,\ldots,\alpha\}\setminus\mathcal{I}_1\]
for some $\kappa\geq2$.\;Then we have
\begin{equation*}
\mathbf{q}(2^{\beta+\kappa})=2^\beta+\varphi(2^{\beta+\kappa})-2^{\beta+\kappa}-2^\beta=-2^{\beta+\kappa-1}.
\end{equation*}
Note that
\begin{equation*}
\begin{aligned}
\mathbf{q}(2^{\beta+\kappa+s-2})&=2^\beta+\sum_{i=\beta+\kappa}^{\beta+\kappa+s-2}\varphi(2^{i})-2^\beta=2^{\beta+\kappa-1}(2^{s-1}-1)>0,
\end{aligned}
\end{equation*}
then $2^{\beta+\kappa+s-2}=0$ from $(\ref{cor1})$,\;so $\beta+\kappa+s-2=\alpha$ and $\mathcal{I}_1=\{\beta+1,\beta+\kappa,\ldots,\alpha\}$.\;In this case,\;
$\mathcal{I}_2=\{\beta+2,\ldots,\beta+\kappa-1\}\neq\emptyset\Rightarrow \kappa>2$,\;but $\mathbf{q}(2^{\beta})=-2\cdot2^{\beta}=-2^{\beta+1}>-2^{\beta+\kappa-1}$,\;this is impossible.\;

Combining {${\mathbf{Case\;1}}$} and {${\mathbf{2}}$},\;we can obtain that $s=|\mathcal{I}_1|=1$.\;Let $\mathcal{I}_1=\{\gamma\}$ with some $\alpha\leqslant\gamma\leqslant\beta+1$.\;If $\mathcal{I}_1=\{\alpha\}$,\;then we have
\[\mathbf{q}(2^{\alpha-2})=-2^{\beta}\neq \mathbf{q}(2^{\alpha-1})=-2^{\beta}-2^{\alpha-1}.\]
This is also a contradiction,\;so $\gamma<\alpha$,\;then
\begin{equation*}
\mathbf{q}(2^{\gamma})=2^{\gamma-1}-2^{\beta}\leqslant0\Rightarrow \gamma-1\leqslant \beta,\gamma\leqslant \beta+1.
\end{equation*}
This gives that $\mathcal{I}_1=\{\beta+1\}$ and therefore
\begin{equation*}
\begin{aligned}
U_X=O_{\beta+1}\cup\left(\mathbb{Z}_{2^\alpha}\setminus 2^{\alpha-\beta}\mathbb{Z}_{2^\alpha}\right)
\end{aligned}
\end{equation*}
from Lemma \ref{l-3.2}.
\end{proof}

\begin{lem}\label{PO}(\cite{PO},Proposition $1.2.12$,\;Corollary $1.2.13$.)\label{l-3.5}\;Let $p$ be a prime and let $G=H\times P$ be an abelian group with a cyclic Sylow $p$-subgroup $P$ of order $p^{s}$.\;Let $P_i$ denote the unique subgroup of order $p^i$ in $G$.\;If $Y\in \mathbb{Z}[G]$ satisfies
\[\chi(Y)\equiv 0 \;\mathrm{mod}\;p^a \]
for all character $\chi$ of order divisible by $p^{s-r}$,\;where $r$ is some fixed number $r\leqslant s$,\;
$r\leqslant a$,\;then there are elements $X_0,X_1,\ldots,X_r,X$ in $\mathbb{Z}[G]$ such that
\[Y=p^a X_0+p^{a-1}P_1X_1+\ldots+p^{a-r}P_rX_r+P_{r+1}X.\]
If $r =\min\{a,s\}$,\;then
\[Y=p^a X_0+p^{a-1}P_1X_1+\ldots+p^{a-r}P_rX_r.\]
Moreover,\;if $Y$ has non-negative coefficients,\;then we can choose the $X_i$ such
that they have non-negative coefficients.
\end{lem}
We now give another version of Lemma \ref{PO}.\;
\begin{lem}\label{l-3.6}Let $p$ be a prime and  $X$ be a multisubset of $\mathbb{Z}_{p^\alpha}$ such that the multiplicity of each element in $X$ at most $p-1$,\;i.e.,\;$\Delta_X(z)\leqslant p-1$ for all $z\in X$.\;If $a\leqslant \alpha-1$ and $X$ satisfies
\[(\mathcal{F}\Delta_X)(z)\equiv0\;\mathrm{mod}\;p^a\]
for all $z$ with $\nu_p(z)\leqslant a$,\;then there is a multiset $S\subseteq \mathbb{Z}_{p^\alpha-a}$ such that
$$X=\psi_{\alpha-a}^{-1}(S),\;$$
i.e.,\;$X$ is a collection of some cosets of the subgroup $p^{\alpha-a}\mathbb{Z}_{p^\alpha}$ in $\mathbb{Z}_{p^\alpha}$.\;
\end{lem}
\begin{proof}We note that
\[\chi_z(\overline{x^X})=(\mathcal{F}\Delta_X)(z).\]
Meawhile,\;the order of the character $\chi_z$ is $p^{\alpha-\nu_p(z)}$ and this order is divisible by $p^{\alpha-a}$ for each $z$ with $\nu_p(z)\leqslant a$.\;Thus the conditions given in this lemma can imply that $\chi(\overline{x^X})\equiv 0 \;\mathrm{mod}(p^a)$ for all characters $\chi$ of order divisible by $p^{\alpha-r}$.\;From Lemma \ref{l-2.3} with $r =\min\{a,\alpha\}=a$,\;there are elements $X_0,X_1,\ldots,X_r$ with non-negative coefficients in $\mathbb{Z}[C_{p^\alpha}]$ such that
\[\overline{x^X}=p^a X_0+p^{a-1}\overline{P_1}X_1+\ldots+p^{a-r}\overline{P_a}X_a=p^a X_0+p^{a-1}\overline{P_1}X_1+\ldots+p\overline{P_{a-1}}X_{a-1}+\overline{P_a}X_a.\]
 It is clear that all the coefficients in $p^{a-j}\overline{P_j}X_j$ are at least $p$ provided $X_j\neq 0$,\;for $0\leqslant j\leqslant a-1$.\;This gives that  $X_0=X_1=\ldots=X_{a-1}=0$ since none of the coefficients in $\overline{x^X}$ exceeds $p-1$.\;Note that $X_a=\overline{x^S}$ for some multiset $S$ of $\mathbb{Z}_{p^\alpha}$ and hence
$\overline{x^X}=\overline{P_a}X_a=x^{p^{\alpha-a}\mathbb{Z}_{p^\alpha}}\overline{x^S}=x^{S+p^{\alpha-a}\mathbb{Z}_{p^\alpha}}$,\;this claims that the multiset $X$ is a collection of some cosets of the subgroup $p^{\alpha-a}\mathbb{Z}_{p^\alpha}$ in $\mathbb{Z}_{p^\alpha}$.\;The  result  follows.\;
\end{proof}
\begin{lem}\label{l-3.7}Let $X$ be a subset of $\mathbb{Z}_{p^\alpha}$ and $0< a\leqslant \alpha$ be a positive integer.\;If $X$ satisfies
\[(\mathcal{F}\Delta_X)(z)=0\]
 for all $z\not\in p^a\mathbb{Z}_{p^\alpha}$,\;then there is a multiset $S\subseteq \mathbb{Z}_{p^\alpha-a}$ such that
$$X=\psi_{\alpha-a}^{-1}(S).\;$$
\end{lem}
\begin{proof}For any $z'\in p^{\alpha-a}\mathbb{Z}_{p^\alpha}$,\;it follows from inverse formula \ref{2.3} that
\begin{equation*}
\begin{aligned}
&\;\;\;\;\;\;p^\alpha(\Delta_X(z+z')-\Delta_X(z))\\
&=\sum_{i\in \mathbb{Z}_{p^\alpha}}(\mathcal{F}\Delta_X)(i)(\zeta_{p^\alpha}^{-i(z+z')}-\zeta_{p^\alpha}^{-iz})\\
&=\left(\sum_{i\in p^a\mathbb{Z}_{p^\alpha}}+\sum_{i\not\in p^a\mathbb{Z}_{p^\alpha}}\right)(\mathcal{F}\Delta_X)(i)(\zeta_{p^\alpha}^{-i(z+z')}-\zeta_{p^\alpha}^{-iz})\\
&=0.
\end{aligned}
\end{equation*}
This shows that $X$ is a union of some cosets of $p^{\alpha-a}\mathbb{Z}_{p^\alpha}$ in $\mathbb{Z}_{p^\alpha}$.\;This proves the lemma.\;
\end{proof}

\section{Directed strongly regular dihedrant}
We now give a characterization of the dihedrant $Dih(n,X,Y)$ to be directed strongly regular.\;Let $\Delta_1=\overline{x^X}+\overline{x^{-X}}=\overline{x^{U_X}}$ and $\Delta_2=\overline{x^Y}\;\overline{x^{-Y}}-\overline{x^X}\;\overline{x^{-X}}$.\;
\begin{lem}\label{l-4.1}The dihedrant $Dih(n,X,Y)$ is a DSRG with parameters $(2n,|X|+|Y|, \mu, \lambda, t)$ if and only if $X$ and $Y$ satisfy the following conditions:
\begin{flalign}
(i)\;&\overline{x^Y}\Delta_1=(\lambda-\mu)\overline{x^Y}+\mu\overline{C_n};\hspace{270pt}\label{4.1}\\
(ii)\;&\overline{x^X}\Delta_1+\Delta_2=\overline{x^X}^2+\overline{x^Y}\;\overline{x^{-Y}}=(t-\mu)e+(\lambda-\mu)\overline{x^X}+\mu\overline{C_n}.\label{4.2}
\end{flalign}
\end{lem}
\begin{proof}Note that
\begin{equation*}
\begin{aligned}
(\overline{x^X}+\overline{x^Ya})^2&=\overline{x^{X}}\;\overline{x^{X}}+\overline{x^X}\;\overline{x^Ya}+\overline{x^Ya}\;\overline{x^X}+\overline{x^Ya}\;\overline{x^Ya}\\
& =\overline{x^{X}}\;\overline{x^{X}}+\overline{x^{Y}}\;\overline{x^{-Y}}+(\overline{x^{X}}\;\overline{x^{Y}}+\overline{x^{Y}}\;\overline{x^{-X}})a\\
&=\overline{x^X}\Delta_1+\Delta_2+(\overline{x^Y}\Delta_1)a.
\end{aligned}
\end{equation*}
Thus,\;from Lemma \ref{l-2.1},\;the dihedrant $Dih(n,X,Y)$ is a DSRG with parameters $(2n,|X|+|Y|, \mu, \lambda, t)$ if and only if
\begin{equation*}
\begin{aligned}
(\overline{x^X}\Delta_1+\Delta_2)+(\overline{x^Y}\Delta_1)a&=te+\lambda(\overline{x^X}+\overline{x^Ya})+\mu(\overline{D_n}-(\overline{x^X}+\overline{x^Ya})-e)\\
&=(t-\mu)e+(\lambda-\mu)\overline{x^X}+\mu\overline{C_n}+((\lambda-\mu)\overline{x^Y}+\mu\overline{C_n})a\\
\end{aligned}
\end{equation*}
This  is equivalent to conditions $(i)$ and $(ii)$.\;
\end{proof}
When $Y=X$,\;we have the following
\begin{lem}\label{l-4.2}The dihedrant $Dih(n,X,X)$ is a DSRG with parameters $(2n,|X|+|Y|, \mu, \lambda, t)$ if and only if $t=\mu$ and $X$ satisfy the following conditions:
\begin{flalign}
\overline{x^X}\Delta_1=(\lambda-\mu)\overline{x^X}+\mu\overline{C_n}\label{4.3}
\end{flalign}
\end{lem}
We now define
\[\mathbf{r}(z)=(\mathcal{F}\Delta_X)(z)=\sum_{i\in X}\zeta_n^{iz}\;\text{and}\;\mathbf{t}(z)=(\mathcal{F}\Delta_Y)(z)=\sum_{i\in Y}\zeta_n^{iz}.\]
Then $\mathbf{r}(z)+\overline{\mathbf{r}}(z)=(\mathcal{F}\Delta_{U_X})(z)$.\;The following lemma  gives a characterization of the dihedrant $Dih(n,X,Y)$ to be directed strongly regular by using $\mathbf{r}(z)$ and $\mathbf{t}(z)$.\;
\begin{lem}\label{l-4.3}The dihedrant $Dih(n,X,Y)$ is a DSRG with parameters $(2n,|X|+|Y|, \mu, \lambda, t)$ if and only if $\mathbf{r}$ and $\mathbf{t}$ satisfy the following conditions:
\begin{flalign}
(i)\;&\mathbf{t}(\mathbf{r}+\overline{\mathbf{r}})=\mu n\Delta_0+(\lambda-\mu)\mathbf{t};\hspace{270pt}\label{4.4}\\
(ii)\;&\mathbf{r}^2+|\mathbf{t}|^2=t-\mu+\mu n\Delta_0+(\lambda-\mu)\mathbf{r}.\label{4.5}
\end{flalign}
\end{lem}
\begin{proof}The equations (\ref{4.1}) and (\ref{4.2}) given in Lemma \ref{l-4.1} are equivalent to
\[(\Delta_Y\ast\Delta_{X\uplus
(-X)})(i)=(\lambda-\mu)\Delta_Y(i)+\mu\Delta_{\mathbb{Z}_n}(i),\]
and $(\Delta_X\ast\Delta_{X})^2(i)+(\Delta_Y\ast\Delta_{-Y})(i)=(t-\mu)\Delta_0(i)+(\lambda-\mu)\Delta_X(i)+\mu\Delta_{\mathbb{Z}_n}(i)$ for any $i\in\mathbb{Z}_n$.\;Then by applying the Fourier transformation on these two equations,\;we can obtain that these two equations are equivalent to
\[\mathbf{t}(\mathbf{r}+\overline{\mathbf{r}})=\mu n\Delta_0+(\lambda-\mu)\mathbf{t},\;\;\mathbf{r}^2+|\mathbf{t}|^2=t-\mu+\mu n\Delta_0+(\lambda-\mu)\mathbf{r}.\]
Then the results follows.\;
\end{proof}
When $Y=X$,\;we have the following results.\;
\begin{lem}\label{l-4.4}The dihedrant $Dih(n,X,X)$ is a DSRG with parameters $(2n,|X|+|Y|, \mu, \lambda, t)$ if and only if $t=\mu$ and the function $\mathbf{r}$  satisfies
\begin{equation}\label{4.6}
\mathbf{r}(\mathbf{r}+\overline{\mathbf{r}})=\mu n\Delta_0+(\lambda-\mu)\mathbf{r}.
\end{equation}
\end{lem}

 Recall in Section 3,\;$\mathbf{q}\overset{\text{def}}=\mathbf{r}+\overline{\mathbf{r}}=\mathcal{F}\Delta_{U_X}$.\;The next result will be needed quite often,\;as it gives a simple description of $X$ if $Dih(p^\alpha,X,X)$ is a directed strongly regular dihedrant.
\begin{lem}\label{l-4.5}Let $Dih(p^\alpha,X,X)$ be a directed strongly regular dihedrant with parameters $(2p^\alpha,2|X|, \mu, \lambda, t)$.\\Then,\;\\
$(a)$\;The function $\mathbf{q}$ satisfies
\begin{equation}\label{4.3.1}
\mathbf{q}(z)=(\mathbf{r}+\overline{\mathbf{r}})(z)=(\mathcal{F}\Delta_{X\uplus(-X)})(z)\in\{0,\lambda-\mu\}
\end{equation}
for any $0\neq z\in \mathbb{Z}_{p^\alpha}$.\;\\
$(b)$\;There are some integers $r_1,r_2,\cdots,r_s$ with $\beta+1\leqslant r_1<r_2<\cdots<r_s\leqslant\alpha$ satisfying
\begin{equation}\label{4.3.2}
\begin{aligned}
U_X=(O_{r_1}\cup O_{r_2}\cup\cdots\cup O_{r_s})\uplus\left(\mathbb{Z}_{p^\alpha}\setminus p^{\alpha-\beta}\mathbb{Z}_{p^\alpha}\right),
\end{aligned}
\end{equation}
for some $\beta$ with $0\leqslant\beta\leqslant\alpha-1$.\\
$(c)$\;If $p$ is an odd prime,\;then
\begin{equation}\label{4.3.3}
\begin{aligned}
U_X=\mathbb{Z}_{p^\alpha}\setminus p^{\alpha-\beta}\mathbb{Z}_{p^\alpha}.
\end{aligned}
\end{equation}
$(d)$\;If $p=2$ and $X\cap(-X)\neq\emptyset$,\;then
\begin{equation}\label{4.3.4}
\begin{aligned}
U_X=O_{\beta+1}\uplus\left(\mathbb{Z}_{2^\alpha}\setminus 2^{\alpha-\beta}\mathbb{Z}_{2^\alpha}\right).
\end{aligned}
\end{equation}
\end{lem}
\begin{proof}
We can get
\begin{equation}\label{4.11}
\overline{\mathbf{r}}(\mathbf{r}+\overline{\mathbf{r}})=\mu p^\alpha\Delta_0+(\lambda-\mu)\overline{\mathbf{r}}
\end{equation}
by taking conjugate on (\ref{4.6}).\;Then the equations (\ref{4.6}) and (\ref{4.11}) gives
\[(\mathbf{r}+\overline{\mathbf{r}})^2=2\mu p^\alpha\Delta_0+(\lambda-\mu)(\mathbf{r}+\overline{\mathbf{r}}),\]
and this implies
\begin{equation*}
(\mathbf{r}+\overline{\mathbf{r}})(z)=\left\{
  \begin{array}{ll}
    2|X|, & \hbox{$z=0$;} \\
    0\text{\;or\;}\lambda-\mu, & \hbox{$z\neq0$;} \\
    \end{array}
\right.
\end{equation*}
Then assertions $(b)$,\;$(c)$ and $(d)$\; follow from Remark 2.2,\;Lemmas \ref{l-3.2},\;\ref{l-3.3} and \ref{l-3.4} respectively.\;
\end{proof}

\section{Some constructions of directed strongly regular dihedrants}
We now give some constructions of directed strongly regular dihedrants $Dih(n;X,Y)$ with $X=Y$ and $X\subset Y$.\;In the following constructions,\;$v$ is a positive divisor of $n$ and $l=\frac{n}{v}$.\;

The directed strongly regular dihedrant $Dih(n,X,X)$  constructed in the following satisfies $X\cap(-X)=\emptyset$.
\begin{con}\label{c-5.1}
Let $v$ be an odd positive divisor of $n$.\;Let $H$ be a subset of $\{1,\cdots,v-1\}\subseteq \mathbb{Z}_n$,\;and $X$ be a subset of $\mathbb{Z}_n$ satisfying the following conditions:\\
$(i)$\;$X=H+v\mathbb{Z}_{n}$.\\
$(ii)$\;$X\cup(-X)=\mathbb{Z}_{n}\setminus v\mathbb{Z}_{n}$.\\
Then $Dih(n,X,X)$ is a DSRG with parameters $\left(2n,n-l,\frac{n-l}{2},\frac{n-l}{2}-l,\frac{n-l}{2}\right)$.\;
\end{con}
\begin{proof}Note that $\overline{x^X}\Delta_1=-l\overline{x^Y}+\frac{n-l}{2}\overline{C_n}$.\;The result follows from Lemma \ref{l-specialdihedrantcharacterization} directly.\;
\end{proof}

The directed strongly regular dihedrant $Dih(n,X,X)$ constructed in the following satisfies  $X\cap(-X)\neq\emptyset$.
\begin{con}\label{c-5.2}Let $v>2$ be an even positive divisor of $n$.\;Let $H$ be a subset of $\{1,\cdots,{v}-1\}\subseteq \mathbb{Z}_n$,\;and $X$ be a subset of $\mathbb{Z}_n$ satisfying the following conditions:\\
$(i)$\;$X=H+v\mathbb{Z}_{n}$.\\
$(ii)$\;$X\cup(-X)=(\mathbb{Z}_{n}\setminus v\mathbb{Z}_{n})\uplus(\frac{v}{2}+v\mathbb{Z}_{n})$.\\
$(iii)$\;$X\cup(\frac{v}{2}+X)=\mathbb{Z}_{n}$.\\
Then $Dih(n,X,X)$ is a DSRG with parameters $\left(2n,n,\frac{n}{2}+l,\frac{n}{2}-l,\frac{n}{2}+l\right)$.\;
\end{con}

\begin{proof}Note that $|X|=\frac{n}{2}$ and $\Delta_1=\overline{C_n}-\overline{x^{v\mathbb{Z}_{n}}}+\overline{x^{\frac{v}{2}+v\mathbb{Z}_{n}}}$.\;Thus $\overline{x^X}\Delta_1=-l\overline{x^X}+\frac{n}{2}\overline{C_n}+\overline{x^X}\;\overline{x^{\frac{v}{2}+v\mathbb{Z}_{n}}}=-l\overline{x^X}+
\frac{n}{2}\overline{C_n}+l\overline{x^{\frac{v}{2}+X}}=-l\overline{x^X}+\frac{n}{2}\overline{C_n}+l\overline{C_n}-l\overline{x^X}=(\frac{n}{2}+l)\overline{C_n}-2l\overline{x^X}$.\;The result follows from Lemma \ref{l-specialdihedrantcharacterization} directly.\;
\end{proof}
\begin{rmk}For $v=2$,\;the dihedrant $Dih(n,X,X)$ is not a genuine DSRG which has parameters $\left(2n,n,n,0,n\right)$,\;we don't consider this case in this paper.
\end{rmk}

We now give some constructions of directed strongly regular dihedrants $Dih(n;X,Y)$ with  $X\subset Y$.\;
\begin{con}\label{c-5.3}
Let $v$ be an odd positive divisor of $n$.\;Let $H$ be a subset of $ \{0,1,\cdots,v-1\}\subseteq \mathbb{Z}_n$ with $0\in H$,\;and $X,Y\subseteq\mathbb{Z}_n$ satisfying the following conditions:\\
$(i)$\;$Y=H+v\mathbb{Z}_{n}=X\cup \{0\}$.\\
$(ii)$\;$Y\cup(-Y)=\mathbb{Z}_{n}\uplus v\mathbb{Z}_{n}$.\\
Then $Dih(n;X,Y)$ is a DSRG with parameters $(2n,n+l-1,\frac{n+l}{2},\frac{n+3l}{2}-2,\frac{n+3l}{2}-1)$.
\end{con}
\begin{proof}We have $|Y|=|X|+1=\frac{n+l}{2}$,\;$\Delta_1=\overline{C_n}+\overline{x^{v\mathbb{Z}_{n}}}-2e$ and
$\Delta_2=\Delta_1+e=\overline{C_n}+\overline{x^{v\mathbb{Z}_{n}}}-e$.\;Thus $\overline{x^Y}\Delta_1=(l-2)\overline{x^Y}+\frac{n+l}{2}\overline{C_n}$ and $\overline{x^X}\Delta_1+\Delta_2=(l-1)e+(l-2)\overline{x^X}+\frac{n+l}{2}\overline{C_n}$.\;The result follows from Lemma \ref{l-generaldihedrantcharacterization} directly.\;
\end{proof}
\begin{con}\label{c-5.4}Let $v$ be an odd positive divisor of $n$.\;Let $H$ be a subset of $ \{0,1,\cdots,v-1\}\subseteq \mathbb{Z}_n$ with $0\in H$,\;and  $X,Y\subseteq\mathbb{Z}_n$ satisfying the following conditions:\\
$(i)$\;$Y=H+v\mathbb{Z}_{n}=X\cup v\mathbb{Z}_{n}$.\\
$(ii)$\;$Y\cup(-Y)=\mathbb{Z}_{n}\uplus v\mathbb{Z}_{n}$.\\
Then $Dih(n;X,Y)$ is a DSRG with parameters $\left(2n,n,\frac{n+l}{2},\frac{n-l}{2},\frac{n+l}{2}\right)$.\;
\end{con}
\begin{proof}We have $|Y|=|X|+l=\frac{n+l}{2}$,\;$\Delta_1=\overline{C_n}-\overline{x^{v\mathbb{Z}_{n}}}$ and
$\Delta_2=\overline{x^{v\mathbb{Z}_{n}}}\Delta_1+l\overline{x^{v\mathbb{Z}_{n}}}=l\overline{C_n}$.\;Thus $\overline{x^Y}\Delta_1=-l\overline{x^Y}+\frac{n+l}{2}\overline{C_n}$ and $\overline{x^X}\Delta_1+\Delta_2=-l\overline{x^X}+\frac{n+l}{2}\overline{C_n}$.\;The result follows from Lemma \ref{l-generaldihedrantcharacterization} directly.\;
\end{proof}

\begin{rmk}
The Lemma 2.4 in \cite{MI} asserts that the dihedrants $Dih(n;X,Y)$ and $Dih(n;bX,b'+bY)$ are isomorphic for any $b\in \mathbb{Z}_n^{\ast}$,\;$b'\in\mathbb{Z}_n$.\;Hence,\;indeed,\;we construct the directed strongly regular dihedrants $Dih(n;bX,b'+bY)$ for any $b\in \mathbb{Z}_n^{\ast}$,\;$b'\in\mathbb{Z}_n$.\;
\end{rmk}

\section{The characterization of directed strongly regular dihedrant $Dih(p^\alpha,X,X)$}

We now can give a characterization of directed strongly regular dihedrant $Dih(p^\alpha,X,X)$ with $p>2$.\;
\begin{thm}\label{t-6.1}For an odd prime $p$ and a positive integer $\alpha$,\;then the dihedrant $Dih(p^\alpha,X,X)$ is a DSRG if and only if
$X=\psi_{\gamma}^{-1}(H)$ for some $1\leqslant\gamma\leqslant\alpha$ and a subset $H\subseteq \mathbb{Z}_{p^\gamma}$ satisfying the following conditions:\\
(i)\;$H\cup (-H)=\mathbb{Z}_{p^\gamma}\setminus\{0\}$.\\
(ii)\;$H\cap (-H)=\emptyset$.
\end{thm}

\begin{proof}It follows from Construction \ref{c-5.1} that the dihedrant $Dih(p^\alpha,X,X)$ with conditions $(i)$ and $(ii)$ is a DSRG\;(where $n=p^\alpha$ and $v=p^{\alpha-\beta}$).\;Conversely,\;suppose that the Cayley digraph $Dih(p^\alpha,X,X)$ is a DSRG with parameters $(2n,k=2|X|,\mu, \lambda, t)$.\;From Lemma \ref{l-4.5} $(c)$,\;we can get
\[X\uplus(-X)=\mathbb{Z}_{p^\alpha}\setminus p^{\alpha-\beta}\mathbb{Z}_{p^\alpha},\]
for some $0\leqslant\beta\leqslant\alpha$,\;proving $(ii)$.\;In this case,\;
\begin{equation*}
\begin{aligned}
\mathbf{r}(z)+\overline{\mathbf{r}}(z)=p^\alpha\Delta_0(z)-p^{\beta}\Delta_{p^{\beta}\mathbb{Z}_{p^\alpha}}(z)
\end{aligned}
\end{equation*}
and hence equation (\ref{4.6}) becomes
\begin{equation*}
\begin{aligned}
p^\alpha\Delta_0(z)\mathbf{r}(z)-p^{\beta}\Delta_{p^{\beta}\mathbb{Z}_{p^\alpha}}(z)\mathbf{r}(z)=\mu n\Delta_0(z)+(\lambda-\mu)\mathbf{r}(z).\;
\end{aligned}
\end{equation*}
This implies that
\[\mathbf{r}(z)=0,\;\forall z\not\in p^{\beta}\mathbb{Z}_{p^\alpha}.\]
From lemma \ref{l-3.7},\;there is a multiset $H\subseteq \mathbb{Z}_{p^{\alpha-\beta}}$ such that
$$X=\psi_{\alpha-\beta}^{-1}(H).\;$$
Let $\gamma=\alpha-\beta$,\;then $X\uplus(-X)=\mathbb{Z}_{p^\alpha}\setminus p^{\alpha-\beta}\mathbb{Z}_{p^\alpha}$ implies $(i)$ and $(ii)$.\;
\end{proof}

We now focus on directed strongly regular dihedrant $Dih(2^\alpha,X,X)$.\;We now prove the non-existence of directed strongly regular dihedrant $Dih(2^\alpha,X,X)$ with $X\cap(-X)=\emptyset$ first.

\begin{lem}\label{l-6.2}A DSRG cannot be a dihedrant $Dih(2^\alpha,X,X)$ with $X\cap(-X)=\emptyset$.\;
\end{lem}
\begin{proof}Suppose $X\cap(-X)=\emptyset$.\;Then From Lemma \ref{l-4.5} $(b)$,\;we also have
\[X\uplus(-X)=O_{\beta+1}\cup O_{\beta+2}\cup\ldots\cup O_{\alpha}=\mathbb{Z}_{2^\alpha}\setminus 2^{\alpha-\beta}\mathbb{Z}_{2^\alpha},\]
Similar to the proof of Lemma \ref{t-6.1},\;there is a multiset $H\subseteq \mathbb{Z}_{2^{\alpha-\beta}}$ such that
$X=\psi_{\alpha-\beta}^{-1}(H)$.\;This gives that $2|X|=2|H|2^{\beta}=2^\alpha-2^{\beta}$ and hence $2|H|=2^{\alpha-\beta}-1$,\;which is impossible.\;
\end{proof}

\begin{thm}\label{t-6.3}For a positive integer $\alpha$,\;the dihedrant $Dih(2^\alpha,X,X)$ is a DSRG if and only if
$X=\psi_{\gamma}^{-1}(H)$ for some $2\leqslant\gamma\leqslant\alpha$ and a subset $H\subseteq \mathbb{Z}_{p^\gamma}$ satisfying the following conditions:\\
(i)\;$H\uplus (-H)=(\mathbb{Z}_{2^\gamma}\setminus\{0\})\uplus\{2^{\gamma-1}\}$.\\
(ii)\;$H\cap (2^{\gamma-1}+H)=\mathbb{Z}_{2^\gamma}$.
\end{thm}
\begin{proof}It follows from Construction \ref{c-5.2} that the dihedrant $Dih(2^\alpha,X,X)$ with the conditions $(i)$ and $(ii)$ is a DSRG.\;Conversely,\;suppose that the dihedrant  $Dih(2^\alpha,X,X)$ is a DSRG with parameters $(2n,k=2|X|,\mu, \lambda, t)$,\;then $X\cap(-X)\neq\emptyset$.\;Then from Lemma \ref{l-4.5} $(d)$ and equation (\ref{2.7}),\;we can get
\begin{equation*}
\begin{aligned}
\mathbf{q}(z)&=(\mathbf{r}+\overline{\mathbf{r}})(z)(\mathcal{F}\Delta_{O_{\beta+1}})(z)+\sum_{i=\beta+1}(\mathcal{F}\Delta_{O_i})(z)\\
&=(\mathcal{F}\Delta_{O_{\beta+1}})(z)+(\mathcal{F}\Delta_{\mathbb{Z}_{2^\alpha}})(z)-(\mathcal{F}\Delta_{\mathbb{Z}_{2^{\alpha-\beta}2^\alpha}})(z)\\
&=\mu\left(\frac{2^{\beta+1}}{(2^{\beta+1},z)}\right)\frac{\varphi(2^{\beta+1})}{\varphi\left(\frac{2^{\beta+1}}{(2^{\beta+1},z)}\right)}+2^\alpha\Delta_0(z)-2^\beta\Delta_{2^\beta\mathbb{Z}_{2^\alpha}}(z)\in\{0,\lambda-\mu\},\\
\end{aligned}
\end{equation*}
for some $0\leqslant\beta\leqslant\alpha-1$,\;and hence
\begin{equation*}
k=\mathbf{q}(0)=2^\alpha,
\mathbf{q}(2^{\beta})=\mu\left(\frac{2^{\beta+1}}{(2^{\beta+1},2^{\beta})}\right)\frac{\varphi(2^{\beta+1})}{\varphi\left(\frac{2^{\beta+1}}{(2^{\beta+1},2^{\beta})}\right)}-2^\beta=-2^{\beta+1}=\lambda-\mu.
\end{equation*}
So from equation (\ref{2.1}),\;we can get $\mu=2^{\alpha-1}+2^\beta$.\;Note that $\mu<k$ and hence $\beta\leqslant\alpha-2$.\;Thus from Lemma \ref{l-4.2},\;equation (\ref{4.6}) becomes
\begin{equation}\label{6.1}
\mathbf{r}(\mathcal{F}\Delta_{O_{\beta+1}}+2^\alpha\Delta_0-2^\beta\Delta_{2^\beta\mathbb{Z}_{2^\alpha}})=\mu {2^\alpha}\Delta_0-2^{\beta+1}\mathbf{r}.
\end{equation}
Note that $(\mathcal{F}\Delta_{O_{\beta+1}})(z)=\mu\left(\frac{2^{\beta+1}}{(2^{\beta+1},z)}\right)\frac{\varphi(2^{\beta+1})}{\varphi\left(\frac{2^{\beta+1}}{(2^{\beta+1},2^{z})}\right)}=0$ for $z\not\in 2^{\beta}\mathbb{Z}_{2^\alpha}$.\;Then the above equation implies $\mathbf{r}(z)=(\mathcal{F}\Delta_X)(z)=0$ for $z\not\in 2^{\beta}\mathbb{Z}_{2^\alpha}$.\;Thus from Lemma \ref{l-3.7},\;there is a multiset $H\subseteq \mathbb{Z}_{2^{\alpha-\beta}}$ such that
$X=\psi_{\alpha-\beta}^{-1}(H)$.\;
So the (\ref{identity}) and equation (\ref{6.1}) implies that,\;for each $z\in \mathbb{Z}_{2^{\alpha-\beta}}$,\;
\begin{equation}\label{6.3}
\begin{aligned}
(\mathcal{F}{\Delta}_{H})(z)(\mathcal{F}{\Delta}_{H}(z)+\overline{\mathcal{F}{\Delta}_{H}(z)})=2^{\alpha-\beta}(1+2^{\alpha-\beta-1}) {\Delta}_0(z)-2(\mathcal{F}{\Delta}_{H}(z))
\end{aligned}
\end{equation}
Let $\gamma=\alpha-\beta$,\;then $\mathcal{F}{\Delta}_{H}(z)+\overline{\mathcal{F}{\Delta}_{H}(z)}\in\{0,-2\}$ for all $0\neq z\in \mathbb{Z}_{2^{\alpha-\beta}}$.\;Hence,\;Lemma \ref{l-3.4} gives that
\begin{equation}\label{5.10}
\begin{aligned}
H\uplus(-H)=O'_{\kappa+1}\cup\left(\mathbb{Z}_{2^\gamma}\setminus 2^{\gamma-\kappa}\mathbb{Z}_{2^\gamma}\right)
\end{aligned}
\end{equation}
for some $0\leqslant \kappa\leqslant \gamma-1$.\;Thus we can get
\begin{equation*}
\begin{aligned}
\mathcal{F}{\Delta}_{H\uplus(-H)}(2^\kappa)=(\mathcal{F}\widetilde{\Delta}_{O'_{\kappa+1}})+2^{\gamma}{\Delta}_0(2^\kappa)
-2^\kappa{\Delta}_{2^\kappa\mathbb{Z}_{2^\gamma}}(2^\kappa)=-2^{\kappa+1}\in\{0,-2\},
\end{aligned}
\end{equation*}
which implies $\kappa=0$.\;So $H\uplus(-H)=(\mathbb{Z}_{2^{\gamma}}\uplus O'_1)\setminus\{0\}=(\mathbb{Z}_{2^{\gamma}}\uplus \{2^{\gamma-1}\})\setminus\{0\}$,\;proving $(i)$.\;Note that $|H|=2^{\gamma-1}$ and (\ref{6.3}) implies that
\[\overline{y^H}(\overline{y^{\mathbb{Z}_{2^\gamma}}}-e+y^{2^{\gamma-1}})=(1+2^{\gamma-1})\overline{y^{\mathbb{Z}_{2^\gamma}}}-2\overline{y^H},\]
which gives that $\overline{y^H}+\overline{y^{2^{\gamma-1}+H}}=\overline{y^{\mathbb{Z}_{2^\gamma}}}$,\;proving $(ii)$.\;
\end{proof}



\section{Characterization of directed strongly regular dihedrant $Dih(p^\alpha,X,Y)$ with $X\subset Y$}
Throughout this section,\; $p$ is an odd prime.\;Let $\mathbf{w}=\mathbf{r}-\mathbf{t}$,\;\;then we have the following lemma.\;Recall that
$O_0,O_1,\ldots,O_\alpha$ are $\mathbb{Z}_{p^\alpha}^\ast$-orbits in $\mathbb{Z}_{p^\alpha}$.\;
\begin{lem}\label{l-7.1}Suppose the dihedrant $Dih(p^\alpha,X,Y)$ is a DSRG with $X\subset Y$.\;Then ${\rm{Im}}{(\mathbf{w})}\subseteq\mathbb{R}$ if and only if $Y\setminus X$ is a union of some $\mathbb{Z}_{p^\alpha}^\ast$-orbits.\;Forthermore,\;let $Dih(p^\alpha,X,Y)$ be a DSRG with $X\subset Y$ and ${\rm{Im}}{(\mathbf{w})}\subseteq\mathbb{R}$,\;then ${\rm{Im}}{(\mathbf{w})}\subseteq\{\rho,\sigma\}$ and $Y\setminus X=O_{r_1}\cup O_{r_1}\cup\ldots\cup O_{r_s}$ for some $0=r_1<r_2<\ldots<r_s\leqslant \alpha$.\;
\end{lem}
\begin{proof}Let the dihedrant $Dih(p^\alpha,X,Y)$ be a DSRG with $X\subset Y$.\;If $Y\setminus X$ is a union of some $\mathbb{Z}_{p^\alpha}^\ast$-orbits,\;then ${\rm{Im}}{(\mathbf{w})}\subseteq\mathbb{R}$ clearly.\;If ${\rm{Im}}{(\mathbf{w})}\subseteq\mathbb{R}$,\;then from equations $(\ref{4.4})$ and $(\ref{4.5})$ in Lemma \ref{l-4.3},\;$\mathbf{w}$ satisfies
\[\mathbf{w}^2=(t-\mu)+(\lambda-\mu)\mathbf{w}.\]
Note that the two eigenvalues ${\rho,\sigma}$ of directed strongly regular dihedrant $Dih(p^\alpha,X,Y)$ are two roots of the quadratic equation $x^2=(t-\mu)+(\lambda-\mu)x$,\;so we can get ${\rm{Im}}{(\mathbf{w})}\in\{\rho,\sigma\}\subseteq \mathbb{Z}$.\;Thus,\;from Lemma \ref{l-2.2},\;we have $\Delta_X-\Delta_Y=\sum_{r=0}^\alpha\alpha_r\Delta_{O_r}$ for some $\alpha_r\in\{0,-1\}$ and $a_0=-1$.\;This implies that $Y\setminus X=O_{r_1}\cup O_{r_1}\cup\ldots\cup O_{r_s}$ for some
$0=r_1<r_2<\ldots<r_s\leqslant \alpha$.
\end{proof}
The following theorem characterize directed strongly regular dihedrant $Dih(p^\alpha,X,Y)$ with $X\subset Y$ and $Y\setminus X$ is a union of some $\mathbb{Z}_{p^\alpha}^\ast$-orbits.

\begin{thm}\label{t-7.2}The dihedrant $Dih(p^\alpha,X,Y)$ is a DSRG with $X\subset Y$ and $Y\setminus X$ is a union of some $\mathbb{Z}_{p^\alpha}^\ast$-orbits if and only if $Y=\psi_{\gamma}^{-1}(H)$,\;$X=Y\setminus\{0\}$ or $X=Y\setminus p^\gamma\mathbb{Z}_{p^\alpha}$
for some $1\leqslant\gamma\leqslant\alpha$ and a subset $H\subseteq \mathbb{Z}_{p^\gamma}$ such that $H\uplus(-H)=\mathbb{Z}_{p^\gamma}\uplus \{0\}$.\;
\end{thm}

\begin{proof}It follows from Construction \ref{c-5.3} that the dihedrant $Dih(p^\alpha,X,Y)$ with condition $(a)$ is a DSRG\;(where $n=p^\alpha$ and $v=p^\gamma$),\;and from Construction \ref{c-5.4} that the dihedrant $Dih(p^\alpha,X,Y)$  with condition $(b)$ is a DSRG\;(where $n=p^\alpha$ and $v=p^\beta$).\;

Now suppose the dihedrant $D_i(p^\alpha,X,Y)$ is a DSRG with $X\subset Y$ and $Y\setminus X$ is a union of some $\mathbb{Z}_{p^\alpha}^\ast$-orbits,\;from Lemma \ref{l-7.1},\;we have ${\rm{Im}}{(\mathbf{w})}\in\{\rho,\sigma\}$ and $Y\setminus X=O_{r_1}\cup O_{r_2}\cup\ldots\cup O_{r_s}$.\;
So $\mathbf{w}=-\sum\limits_{i=1}^s\mathcal{F}\Delta_{O_i}$.\;

We assert $\{0=r_1,r_2,\ldots,r_s\}=\{0,1,\ldots,s-1\}$.\;Assume $s>1$ since this assertion hold for $s=1$ trivially.\;In fact,\;if there is a integer $u$ such that $r_{u+1}>r_{u}+1$ for some $1\leqslant u\leqslant s-1$,\;then
\begin{equation*}
\begin{aligned}
\mathbf{w}(p^{r_u})&=-\sum_{i=1}^u\mu\left(\frac{p^{r_i}}{(p^{r_i},p^{r_u})}\right)\frac{\varphi(p^{r_i})}{\varphi\left(\frac{p^{r_i}}{(p^{r_i},p^{r_u})}\right)}=-\sum_{i=1}^u\varphi(p^{r_i})<0.
\end{aligned}
\end{equation*}
but
\begin{equation*}
\begin{aligned}
\mathbf{w}(p^{r_s})&=-\sum_{i=1}^s\mu\left(\frac{p^{r_i}}{(p^{r_i},p^{r_s})}\right)\frac{\varphi(p^{r_i})}{\varphi\left(\frac{p^{r_i}}{(p^{r_i},p^{r_s})}\right)}=-\sum_{i=1}^s\varphi(p^{r_i})<-\sum_{i=1}^u\varphi(p^{r_i})=\mathbf{w}(p^{r_u})<0,
\end{aligned}
\end{equation*}
a contradiction.\;This shows that $r_{u+1}=r_{u}+1$ for each $1\leqslant u\leqslant s-1$.\;Then
$$Y\setminus X=\bigcup_{i=0}^{s-1}O_i=p^\beta\mathbb{Z}_{p^\alpha},$$
where $\beta=\alpha-s+1$.\;Therefore $\mathbf{r}-\mathbf{t}=\mathbf{w}=-\mathcal{F}(\Delta_{p^\beta\mathbb{Z}_{p^\alpha}})=-p^{\alpha-\beta}\Delta_{p^{\alpha-\beta}\mathbb{Z}_{p^\alpha}}$ and so $\mathbf{r}=\mathbf{t}-p^{\alpha-\beta}\Delta_{p^{\alpha-\beta}\mathbb{Z}_{p^\alpha}}$.\;Then from Lemma \ref{l-4.3},\;equations (\ref{4.4}) and (\ref{4.5}) become
\begin{flalign}
&\mathbf{t}^2+|\mathbf{t}|^2-2p^{\alpha-\beta}\mathbf{t}\Delta_{p^{\alpha-\beta}\mathbb{Z}_{p^\alpha}}=\mu {p^\alpha}\Delta_0+(\lambda-\mu)\mathbf{t},\hspace{140pt}\label{7.1}\\
&(\mathbf{t}-p^{\alpha-\beta}\Delta_{p^{\alpha-\beta}\mathbb{Z}_{p^\alpha}})^2+|\mathbf{t}|^2=t-\mu+\mu {p^\alpha}\Delta_0+(\lambda-\mu)(\mathbf{t}-p^{\alpha-\beta}\Delta_{p^{\alpha-\beta}\mathbb{Z}_{p^\alpha}})\label{7.2}
,\end{flalign}
The difference of these two equations gives
\begin{equation}\label{7.3}
p^{2(\alpha-\beta)}\Delta_{p^{\alpha-\beta}\mathbb{Z}_{p^\alpha}}=t-\mu+(\mu-\lambda)p^{\alpha-\beta}\Delta_{p^{\alpha-\beta}\mathbb{Z}_{p^\alpha}}.
\end{equation}
{${\mathbf{Case\;1}}$}:\;$\beta=\alpha$.\;In this case,\;$Y=X\cup\{0\}$,\;hence $t-\lambda=1$.\;Since the dihedrant $Dih(p^\alpha,X,Y)$ is  a DSRG with parameters $(2n,|X|+|Y|, \mu, \lambda, t)$,\;its complement $Dih(p^\alpha,\mathbb{Z}_n\setminus X\setminus\{0\},\mathbb{Z}_n\setminus Y)$ is also a DSRG.\;Note that $\mathbb{Z}_n\setminus X\setminus\{0\}=\mathbb{Z}_n\setminus Y$ since $Y=X\cup\{0\}$.\;So from Theorem \ref{t-6.1},\;we can get $Y^c=\mathbb{Z}_{p^\alpha}\setminus Y=\psi_{\gamma}^{-1}(H')$ for some $1\leqslant\gamma\leqslant\alpha$ and a subset $H'\subseteq \mathbb{Z}_{p^\gamma}$ such that $H\uplus (-H)=\mathbb{Z}_{p^\gamma}\setminus\{0\}$.\;This implies that
$Y=\psi_{\gamma}^{-1}(H)$,\;where $0\in H=\mathbb{Z}_{p^\gamma}\setminus H'$,\;and $H\uplus(-H)=\mathbb{Z}_{p^\gamma}\uplus \{0\}$.\\
{${\mathbf{Case\;2}}$}:\;$0<\beta<\alpha$.\;In this case,\;we have $\mu-\lambda=p^{\alpha-\beta}$ and $t=\mu$.\;Thus (\ref{7.1}) becomes
\begin{equation}\label{7.4}
\mathbf{t}^2+|\mathbf{t}|^2=\mathbf{t}(\mathbf{t}+\overline{\mathbf{t}})=\mu {p^\alpha}\Delta_0+p^{\alpha-\beta}(2\Delta_{p^{\alpha-\beta}\mathbb{Z}_{p^\alpha}}-1)\mathbf{t}.
\end{equation}
Hence,\;we get
\[(\mathbf{t}+\overline{\mathbf{t}})^2=2\mu {p^\alpha}\Delta_0+p^{\alpha-\beta}(2\Delta_{p^{\alpha-\beta}\mathbb{Z}_{p^\alpha}}-1)(\mathbf{t}+\overline{\mathbf{t}}),\]
which implies
\begin{equation}\label{7.5}
(\mathbf{t}+\overline{\mathbf{t}})(z)=\left\{
  \begin{array}{ll}
    0 \text{\;or\;} p^{\alpha-\beta}, & \hbox{$z\in p^{\alpha-\beta}\mathbb{Z}_{p^\alpha}\setminus\{0\}$;} \\
    0 \text{\;or\;} -p^{\alpha-\beta}, & \hbox{$z\not\in p^{\alpha-\beta}\mathbb{Z}_{p^\alpha}$.}
  \end{array}
\right.
\end{equation}
Thus we have
\[(\mathbf{t}+\overline{\mathbf{t}})(z)=(\mathcal{F}\Delta_{Y\uplus(-Y)})(z)\equiv0\;\mathrm{mod}\;p^{\alpha-\beta},\]
for each $z\neq0$.\;Therefore,\;from Lemma \ref{3.7},\;
$$Y\uplus(-Y)=\psi_{\gamma}^{-1}(H'')$$ for some $1\leqslant\beta\leqslant\alpha$ and a subset $H''\subseteq \mathbb{Z}_{p^\beta}$.\;
We can also write $Y\uplus(-Y)=S+p^{\beta}\mathbb{Z}_{p^\alpha}$ for some $S\subseteq\{0,1,\ldots,p^\beta-1\}$,\;then from equation (\ref{2.4}),\;we have
\begin{equation}\label{7.6}
\begin{aligned}
(\mathbf{t}+\overline{\mathbf{t}})(z)=p^{\alpha-\beta}(\mathcal{F}\Delta_{S})(z)\Delta_{p^{\alpha-\beta}\mathbb{Z}_{p^\alpha}}(z).
\end{aligned}
\end{equation}
Combining equations (\ref{7.4}) and (\ref{7.6}),\;we have
\begin{equation}\label{7.7}
p^{\alpha-\beta}\mathbf{t}(\mathcal{F}\Delta_{S})\Delta_{p^{\alpha-\beta}\mathbb{Z}_{p^\alpha}}=\mu {p^\alpha}\Delta_0+p^{\alpha-\beta}(2\Delta_{p^{\alpha-\beta}\mathbb{Z}_{p^\alpha}}-1)\mathbf{t}.
\end{equation}
Therefore (\ref{7.7}) implies that
\[(\mathcal{F}\Delta_Y)(z)=\mathbf{t}(z)=0.\]
for any $z\not\in p^{\alpha-\beta}\mathbb{Z}_{p^\alpha}$.\;So from Lemma \ref{3.7},\;
$$Y=\psi_{\beta}^{-1}(H)$$ for some subset $H\subseteq \mathbb{Z}_{p^\beta}$.\;
Thus,\;(\ref{identity}),\;and (\ref{7.4}) implies that for each $z\in \mathbb{Z}_{2^{\beta}}$,\;
$$\mathcal{F}{\Delta}_{H}(z)(\mathcal{F}{\Delta}_{H}(z)+\overline{\mathcal{F}{\Delta}_{H}(z)})=\mu {p^{2\beta-\alpha}}{\Delta}_0(z)+\mathcal{F}{\Delta}_{H}(z)$$
hold for all $z\in \mathbb{Z}_{p^\beta}$.\;
 Then $(\mathcal{F}{\Delta}_{H}(z)+\overline{\mathcal{F}{\Delta}_{H}(z)})\in\{0,1\}$ for all $0\neq z\in \mathbb{Z}_{p^{\beta}}$.\;We assert that $H\uplus(-H)=\mathbb{Z}_{p^\beta}\uplus\{0\}$.\;Let $H_1=\mathbb{Z}_{p^\alpha}\setminus H$,\;then
\[(\mathcal{F}{\Delta}_{H_1}(z)+\overline{\mathcal{F}{\Delta}_{H_1}(z)})\in\{0,-1\}.\]
Let $O'_1,O'_1,\ldots,O'_{\beta}$ are $\mathbb{Z}^\ast_{p^{\beta}}$-orbits in $\mathbb{Z}_{p^{\beta}}$.\;It follows from Lemma \ref{l-3.3} that
\begin{equation}\label{5.10}
\begin{aligned}
H_1\uplus(-H_1)=\mathbb{Z}_{p^{\beta}}\setminus p^{\beta-\kappa}\mathbb{Z}_{p^{\beta}}.
\end{aligned}
\end{equation}
for some $0\leqslant \kappa\leqslant \beta-1$.\;\;Thus 
$$(\mathcal{F}{\Delta}_{H_1\uplus(-H_1)})(p^\kappa)=-p^\kappa=-1,$$
which implies that $\kappa=0$.\;So $H_1\uplus(-H_1)=\mathbb{Z}_{p^{\beta}}\setminus\{0\}$ and hence $H\uplus(-H)=\mathbb{Z}_{p^{\beta}}\uplus\{0\}$.\;This completes the proof.\;
\end{proof}

\section*{References}

\bibliographystyle{plain}
\bibliography{12}

\end{document}